\documentclass[12pt]{amsart}
\usepackage{amssymb}
\usepackage{amsmath}
\usepackage{setspace}
\usepackage{mathtools}
\usepackage{verbatim}
\usepackage{lipsum}
\usepackage{amsthm}
\usepackage{tikz-cd}
\usepackage{hyperref}
\usepackage{adjustbox}

\newtheorem{theorem}{Theorem}[section]
\newtheorem{lemma}[theorem]{Lemma}
\newtheorem{definition-lemma}[theorem]{Definition-Lemma}

\newtheorem{remark}[theorem]{Remark}

\newtheorem{claim}{Claim}

\newcommand{\bb}[1]{\mathbb{#1}}

\newcommand{\cal}[1]{\mathcal{#1}}

\newcommand{\Spec}[1]{\operatorname{Spec} \, #1}

\newcommand{\sing}[1]{\operatorname{Sing} \, #1}
\newcommand{\supp}[1]{\operatorname{Supp} \, #1}
\newcommand{\exc}[1]{\operatorname{Exc} \, #1}

\newcommand{\Null}[1]{\operatorname{Null} \, #1}

\title{Foliated minimal models and flops}

\author{Paolo Cascini}

\author{Roktim Mascharak}
\author{Calum Spicer}

\address{Department of Mathematics, Imperial College London, 180 Queen's
Gate, London SW7 2AZ, UK}
\email{p.cascini@ic.ac.uk}

\address{Department of Mathematics, King's College London, Strand,
London WC2R 2LS, UK}
\email{roktim.mascharak@kcl.ac.uk}

\address{Department of Mathematics, King's College London, Strand,
London WC2R 2LS, UK}
\email{calum.spicer@kcl.ac.uk}

\begin{document}

\begin{abstract}
We study minimal models and flops for foliations. We show that if 
$\mathcal F$ is a rank one foliation with canonical singularities on a 
normal projective $\mathbb Q$-factorial variety and $K_{\mathcal F}$ is 
pseudo-effective, then any two outputs of the $K_{\mathcal F}$-MMP are 
isomorphic.

For co-rank one foliations on threefolds, we prove existence results for 
$D$-flops in the klt setting and, under additional hypotheses, in the 
F-dlt setting. By contrast, we construct examples showing that rank one 
foliations display pathologies absent from the classical MMP: flopping 
contractions need not admit $D$-flops, and nef and big canonical divisors 
need not give rise to canonical models, even in the category of algebraic 
spaces.
\end{abstract}

\maketitle

\tableofcontents

\section{Introduction}

A fundamental principle in the birational geometry of varieties is that
minimal models are unique up to flops.  More precisely, Kawamata proved
that two birational minimal models of the same pair are connected by a
finite sequence of flops \cite{Kawamata08}.  Thus, although minimal
models are not unique in general, their non-uniqueness is controlled by
crepant birational transformations which are isomorphisms in codimension
one.

The aim of this paper is to study the analogous question for foliations.
The minimal model programme for foliations is now known in several
important cases, especially for rank one foliations and for co-rank one
foliations on threefolds; see for example \cite{spicer20,CS21,CS20}.  Once
existence of minimal models is known, it is natural to ask whether the
outputs of the foliated MMP are again connected by flops.  In the
co-rank one case this question has been studied by Jiao--Voegtli
\cite{JV23}, who prove that two minimal models descending from a common
threefold pair endowed with an F-dlt co-rank one foliation of general type are connected
by a sequence of flops.  Related developments in the foliated MMP and in
the study of flops for foliations appear in 
\cite{ChaudhuriMascharak24,ChenLiuWang25}.

\medskip

In this paper, we are interested in the birational geometry of foliations from two
complementary points of view.  First, we show that rank one foliations
behave in a markedly different way: their minimal
models are much more rigid, while their flopping contractions may exhibit
pathologies which do not occur in the classical MMP.  Secondly, we prove
existence results for flops of co-rank one foliations on threefolds.

\medskip 

We begin with rank one foliations, which admit a strong uniqueness
statement for minimal models:

\begin{theorem}\label{t_uniqueness_min_models}
Let $X$ be a normal projective variety with $\mathbb Q$-factorial
singularities, and let $\mathcal F$ be a rank one foliation on $X$ with
canonical singularities such that $K_{\mathcal F}$ is pseudo-effective.
Let
\[
    f_1\colon X\dashrightarrow Y_1
    \qquad\text{and}\qquad
    f_2\colon X\dashrightarrow Y_2
\]
be two outputs of a $K_{\mathcal F}$-MMP.  

Then there exists an
isomorphism
\(
    \phi\colon Y_1\to Y_2
\)
such that $f_2=\phi\circ f_1$.
\end{theorem}

Thus, for rank one foliations with canonical singularities, minimal
models are unique as outputs of the MMP.  A technical ingredient in the proof of Theorem
\ref{t_uniqueness_min_models} is a Bertini-type statement for invariant
subvarieties.  In its simplest form, it says that for a smooth rank one
foliation, one can choose a sufficiently ample divisor which contains no
invariant subvariety.  We prove a more general version which applies to
singular rank one foliations (cf. Theorem \ref{t_Bertini}). More specifically, 
if $X$ is a normal projective variety of dimension $n$, $\mathcal F$ is a
rank one foliation with canonical singularities, and $L$ is ample, then,
after replacing $L$ by a multiple, there is $A\in |L|$ which contains no
log canonical centre of $\mathcal F$ not contained in
$\operatorname{Sing}^+\mathcal F$.  The proof uses foliated jets: for a
general section $s\in H^0(X,L)$, its foliated $n$-jet
\(j^n_{\mathcal F}(s)\)
is nowhere vanishing on the locus where the foliation is regular and
Gorenstein.  If an invariant subvariety were contained in $\{s=0\}$,
then all the derivatives of $s$ along the foliation would vanish along that
subvariety, contradicting the choice of $s$.

\medskip

We then turn to co-rank one foliations.  Our main existence theorem is
the following:
\begin{theorem}
\label{thm_dlt_flop}
Let $X$ be a smooth threefold and let $\cal F$ 
be an F-dlt foliation of co-rank one on $X$. 
Let
$f\colon X \rightarrow Z$ be a $K_{\mathcal F}$-flopping contraction and let  $D$ be a $\mathbb Q$-Cartier divisor on $X$ such that 
$-D$ is $f$-ample.   Let $S_1, \dots, S_N \subset \hat{X}$ denote the collection of all the separatrices of $\mathcal F$ (formal or otherwise) which meet $C:= \exc f$, where $\hat{X}$ is the formal completion of $X$ along $C$.
Suppose that 

\begin{enumerate}
\item $C$ is irreducible; 

\item $C \subset S_1$ and $S_1$ is normal (in particular $C$ is tangent to $\mathcal F)$; and

\item  $-S_1$ is $f$-ample.
\end{enumerate}

Then the $D$-flop exists (see \S \ref{s_notation}).
\end{theorem}

The proof of Theorem \ref{thm_dlt_flop} uses the existence of  separatrices of
non-dicritical
co-rank one foliations and a reduction to the construction of a klt
ample model.  The idea is that by working formally along the exceptional
locus one can use the separatrices of the foliation to produce a boundary whose log MMP produces the desired
$D$-ample model.  This is similar to how the flip is constructed in \cite{CS21}. However, in the case that the flopping curve is contained in the singular locus of the foliation, the techniques of \cite{CS21} do not immediately imply the existence of the flop.  In this situation a more refined understanding of the dynamics of the foliation near the flopping curve is required: by analysing the Camacho-Sad indices around the flopping curve, we are able to demonstrate the existence of a separatrix which meets the flopping curve properly, and this extra separatrix provides the additional flexibility needed to make the techniques of \cite{CS21} apply here.

We remark that Theorem \ref{thm_dlt_flop} is also surprising because we are able to prove a stronger statement than would be expected from the setting of pairs.  Indeed, flops do not necessarily exist for dlt pairs (we refer to Remark \ref{dlt_flop_nonexist} for a further comment on this point).

Using similar ideas we show the existence of flops for klt foliations.

\begin{theorem}
\label{thm_flop_klt}
Let $X$ be a klt quasi-projective threefold, let $\mathcal F$ be a foliation of co-rank one on $X$ and let $\Delta \ge 0$ be a $\mathbb Q$-divisor such that 
$(\mathcal F, \Delta)$ is klt. 
Let $f\colon X \rightarrow Z$ be a small contraction such that $K_{\mathcal F}+\Delta$ is $f$-numerically trivial. 
Let $D \ge 0$ be a $\mathbb Q$-Cartier divisor such that $-D$ is $f$-ample.

Then the $D$-flop exists.
\end{theorem}

\medskip 

Finally, we show that rank one foliations also display 
pathologies from the point of view of flops and canonical models.  Even
when a rank one foliation has canonical singularities and admits a
flopping contraction, the corresponding $D$-flop need not exist (cf. Theorem \ref{t_flopsdonotexist}).

We also construct examples showing that the basepoint-free 
theorem (which already fails to hold for foliations on surfaces) can fail in more dramatic ways than would be expected from the 
failure for surface foliations.
 More precisely, we construct a smooth
projective threefold endowed with a rank one foliation with canonical
singularities whose canonical divisor is nef and big, but whose null
locus cannot be contracted, even in the category of algebraic spaces (cf. Theorem \ref{t_canonicalmodelsdonotexist}).
These examples are found by first choosing an appropriate $\mathbb P^1$-bundle over a Hilbert modular surface, $P \to X$, such that one of the tautological foliations on the Hilbert modular surface lifts to $P$.  We then ``perturb'' this lift slightly so that it acquires a tangency with the $\mathbb P^1$-bundle structure along a divisor.  The null locus of this foliation will then be a copy of $X$ and the restriction of our foliation to this copy of $X$ is precisely one of the tautological foliations.  As is now well-known, these foliations are non-abundant, and this obstructs the contraction of the null locus.

\subsection{Acknowledgements}  
We would like to thank Federico Bongiorno, Enrica Floris and Jihao Liu for several useful discussions on the content of this paper. 
The first author is partially funded by a Simons Collaboration grant. The second author is supported by the EPSRC Centre for Doctoral Training in Geometry and Number Theory (LSGNT). The third author is partially funded by EPSRC.

\section{Preliminaries}

We work over the field of complex numbers $\mathbb C$.

\subsection{Notations} \label{s_notation}
Let $X$ be a normal projective variety and let $M$ be a $\mathbb Q$-Cartier $\mathbb Q$-divisor on $X$. We define the {\bf null locus} of $M$ as
\[
\Null M\coloneqq \bigcup V,
\]
where the union runs over all the subvarieties $V\subset X$ such that $M|_V$ is not big.

\medskip 

Let $X$ be a normal variety of dimension $n$.  A {\bf foliation $\mathcal F$ of rank $r$} (or, equivalently, of {\bf co-rank $n-r$}) on $X$ is a coherent subsheaf $T_{\cal F}\subset T_X$ of rank $r$ which is saturated and closed under Lie bracket. A subvariety $S\subset X$ is said to be {\bf $\cal F$-invariant} if, 
in a neighbourhood  $U$ of the generic point of $S$, $T_{\mathcal F}$ is locally free and for  any section $\partial \in H^0(U,T_{\mathcal F})$, we have that $\partial (\mathcal I_{S\cap U})\subset \mathcal I_{S\cap U}$, where $\mathcal I_{S\cap U}$ denotes the ideal sheaf of $S\cap U$ in $U$. 
We say that a subvariety $V \subset X$ is {\bf tangent to} $\mathcal F$
if, at the general point of $V$, we have
\(
T_V \subset T_{\mathcal F}|_V.
\)

If $\pi\colon Y\dashrightarrow X$ is a dominant map between normal varieties and $\cal F$ is a foliation on $X$, we denote by $\pi^{-1}\cal F$ the {\bf induced foliation} on $Y$ (e.g. see \cite[\S 2.3]{CS21}). 
We refer to \cite[\S 2.4 and \S 3.3]{CS21} and \cite[\S 2.3]{CS20} for the standard definitions of foliated singularities (e.g. canonical, log canonical, F-dlt, non-dicritical) and we refer to \cite[\S 2.6]{CS21} for some standard definitions in the Minimal Model Program. 
In particular, if \(X\) is a normal projective variety and  \(\mathcal F\) is a foliation on \(X\) such that \(K_{\mathcal F}\) is \(\mathbb Q\)-Cartier, then 
a {\bf \(K_{\mathcal F}\)-minimal model} of \(\mathcal F\) is a $K_{\mathcal F}$-negative birational
contraction
\(
    \varphi:X\dashrightarrow Y
\)
to a normal projective variety \(Y\), such that if $\mathcal F_Y$ is the induced foliation on $Y$ then $K_{\mathcal F_Y}$ is nef.

We say that a normal variety $X$ is {\bf potentially klt} if there exists an
effective $\mathbb Q$-divisor $\Delta$ on $X$ such that $(X,\Delta)$ is klt.

\medskip

By a \textbf{flopping contraction} for a
foliated pair $(\mathcal F,\Delta)$ on a normal variety $X$ we mean a
projective birational morphism
\(  f\colon X\to Z \)
such that
\begin{enumerate}
    \item $f_*\mathcal O_X=\mathcal O_Z$;
    \item $\operatorname{Exc}f$ is one-dimensional; and
    \item $K_{\mathcal F}+\Delta\equiv_f 0$.
\end{enumerate}
Suppose moreover that 
 $D$ is a $\mathbb Q$-Cartier 
divisor on $X$ such that $-D$ is $f$-ample.  A \textbf{$D$-flop} of
$f$ is a birational map
\(
    \phi\colon X\dashrightarrow X^+
\) over $Z$
which is an isomorphism in codimension one and such that, if
$D^+:=\phi_*D$ and $f^+\colon X^{+}\to Z$ is the induced morphism, then $X^+$ is a normal variety, $D^+$ is an $f^+$-ample $\mathbb Q$-Cartier divisor.  Note that we do not assume that
$\rho(X/Z)=1$.

If $X$ is a smooth threefold and $\cal F$ is a co-rank one foliation on $X$ then \cite[Lemma 2.14]{CS21} implies that $\cal F$ is non-dicritical if and only if it is strongly non-dicritical, i.e. for any sequence of blow-ups 
\[
X_n\to \dots \to X_1\to X
\]
in smooth centres tangent to $\cal F_i$ or smooth centres contained in $\sing \cal F_i$, where $\cal F_i$ is the induced foliation  on $X_i$, we have that every component of the exceptional locus of $X_n\to X$ is $\cal F_n$-invariant. 
Moreover,  by \cite[Theorem 11.3]{CS21}, it follows that if $X$ is a threefold and $(\mathcal F,\Delta)$ is a co-rank one foliated pair which is  F-dlt or canonical, then $\mathcal F$ has non-dicritical singularities. Similarly, 
\cite[Corollary III.i.4]{MP13} implies that if $X$ is a normal variety and $\cal F$ is a rank one foliation with canonical singularities, then $\cal F$ has non-dicritical singularities.

\subsection{Extending Separatrices}

For the reader's convenience we recall some facts relating to the extension (or prolongation) of (formal) separatrices of co-rank one foliations.  The basic ideas are due to \cite{CC92}, but we will recall variants of their ideas in our set-up found in \cite{spicer20} and \cite{CS21}.  

Let $X$ be a quasi-projective threefold, let $\mathcal F$ be a foliation of co-rank one on $X$ and suppose that $\mathcal F$ is F-dlt. In particular, $\mathcal F$ is non-dicritical (cf.  \cite[Theorem 11.3]{CS21}).
Let $C$ be a curve in $X$ which is tangent to $\mathcal F$. 
Then, applying \cite[Lemma 3.14]{CS21} to the
pair \((\mathcal F,0)\), there exists, at a general point of \(C\), a
(possibly formal) separatrix whose support contains the germ of \(C\).
Indeed, if $C$ is not contained in $\sing \mathcal F$ then  this separatrix is unique, while if $C$ is contained in $\sing \mathcal F$ then the foliation is simple at the general point of \(C\) (cf. \cite[Lemma 3.8]{CS21}),
and we can take one of the corresponding local separatrices.

A priori, given a closed point $x \in C$, a separatrix $S$ can only be defined at the formal completion of $X$ at $x$.  However, thanks to \cite[Lemma 3.5]{CS21} we may extend $S$ to a formal subscheme of $\hat{X}$, the completion of $X$ along $C$.
Thus, in this situation we may without ambiguity refer to the set of all separatrices which meet $C$ as being a set of formal subschemes of $\hat{X}$.

\subsection{Adjunction type results in the co-rank one case}

\begin{lemma}
\label{lem_nef_comparison_lem}
Let $X$ be a normal threefold and let $\cal F$ 
be a foliation of co-rank one on $X$.  Suppose that 
\begin{enumerate}
\item $X$ is potentially klt;
\item $K_{\mathcal F}$ is $\mathbb Q$-Cartier;
\item $\mathcal F$ has non-dicritical singularities; and 
\item $\mathcal F$ has F-dlt singularities away from a closed subset of codimension at least $3$.
\end{enumerate}
Let $C$ be an irreducible curve tangent to $\mathcal F$ and let $\widehat X$ be the formal completion of $X$ along $C$. Let
$T_1, \dots, T_m$ be the collection of all $\mathcal F$-invariant divisors 
in $\widehat X$ meeting $C$.
Suppose that $K_{\widehat X}+\sum_{i = 1}^m T_i$ is $\bb Q$-Cartier.

Then, $(K_{\cal F} - (K_{\widehat X}+\sum T_i))\cdot C \geq 0$.
\end{lemma}
\begin{proof}

Perhaps relabelling the $T_i$ we may assume that $C \subset T_1$. Let $\nu\colon T_1^\nu \rightarrow T_1$ be the normalisation and let $C'$ be an irreducible component of $\nu^{-1}(C)$.
If 
$C \subset \sing \mathcal F$, then 
\cite[Lemma 3.8]{CS21} implies that  $\cal F$ has simple singularities at the generic 
point of $C$. Furthermore, if the simple singularity is of type (2) 
(cf. \cite[Definition 2.8]{CS21})
then we may  further assume that
$T_1$ is a strong separatrix (cf. \cite[Lemma 2.20]{CS21}). 

Let $M \coloneqq (K_{\cal F} - (K_{\widehat X}+\sum_{i = 1}^m T_i))\vert_{T_1^\nu}$
and
let us write
\[
K_{\cal F}|_{T_1^\nu}=K_{T_1^{\nu}}+\Theta\qquad \text{and}\qquad  
(K_{\widehat X}+\sum_{i =1}^m T_i)|_{T_1^{\nu}}=K_{T_1^{\nu}}+\Theta'
\]
then \cite[Corollary 3.20]{CS21}
 implies that $0\le \Theta'\le \Theta$ and if $C\subset \sing X$ then $m_{C'}\Theta=m_{C'}\Theta'$. 
In particular, we have that  $M \equiv \Theta-\Theta' \geq 0$. (We remark that the cited Corollary requires that $T_i$ is $\mathbb Q$-Cartier for all $i$, 
however, this is not necessary. Indeed, our claimed inequality is local about the generic 
point of $C$, and so by restricting to a neighbourhood of the generic point of $C$, 
we see that $T_i$ is $\mathbb Q$-Cartier since $X$ is potentially klt).

If we can show that 
$m_{C'}(\Theta-\Theta') = 0$ then we may conclude that $M\cdot C' \geq 0$.  As observed earlier, $m_{C'}(\Theta-\Theta') = 0$ if $C \subset \sing X$, so we may freely assume that $C$ is not contained in the singular locus of $X$.
We argue in cases based on whether $C \subset \sing \mathcal F$ or not.

If $C$ is not contained in $\sing \mathcal F$, 
then \cite[Proposition 3.14]{CS23} implies that $m_{C'}\Theta=0$,
 and so $m_{C'}(\Theta-\Theta') = 0$ as required.

If $C$ is contained in $\sing \cal F$, then by the
choice of \(T_1\) above and the local coefficient computation in 
\cite[Corollary 3.20]{CS21}, we have that 
$m_{C'}\Theta = m_{C'}\Theta' = 1$.  Hence $m_{C'}(\Theta-\Theta') = 0$, as required. 
\end{proof}

\section{Minimal Models of rank one foliations}
The goal of this Section is to prove Theorem \ref{t_uniqueness_min_models}.

We begin by recalling the construction of the vector bundle $J^m_{\cal{F}}L$ of $m$-jets of a line bundle $L$ along a foliation $\cal{F}$ and the Atiyah sequence associated to it: 

\begin{definition-lemma}\label{Jets 1}
	Let $X$ be a normal variety, let $\mathcal F$ be a rank one foliation on $X$ with Gorenstein singularities, and let $L$ be a line bundle on $X$. 
    
    Then, for every $m\ge 0$, there exists a vector bundle $J^m_{\mathcal F}L$ of rank $m+1$ on $X$ such that
    $J^0_{\mathcal F}L\coloneqq L$ and, for every $m\ge 1$, $J^m_{\mathcal F}L$ fits into the exact sequence
\[
0\rightarrow \mathcal{O}_X(mK_{\mathcal F})\otimes L
\rightarrow J^m_{\mathcal F}L
\rightarrow J^{m-1}_{\mathcal F}L
\rightarrow 0.
\]

Moreover, there exists a $\mathbb C$-linear natural  map
\[
    j^m_{\mathcal F}\colon L\longrightarrow J^m_{\mathcal F}L,
\]
such that, for any section $s$ of $L$ and at every point $x\in X$, 
after shrinking around $x$, we may choose a  generator $\partial$
of $\mathcal F$ and a generator $\alpha$ of $L$ such that, if we write 
$s=\tilde{s}\alpha$ with $\tilde s\in \mathcal O_X$, then
\[
j_{\mathcal F}^m(s)=
(\tilde{s},\partial \tilde{s},\dots,\partial^m\widetilde{s})
\]
in the induced trivialisation of $J^m_{\mathcal F}L$.

We will refer to $J^m_{\mathcal F}L$ as the {\bf vector bundle of foliated $m$-jets of $L$ along $\mathcal F$} and to $j_{\mathcal F}^m(s)$ as the {\bf foliated $m$-jet of $s$}. 
\end{definition-lemma}

\begin{proof}

We follow the construction of \(J^m_{\mathcal F}L\) in \cite[\S 6.2]{pereira2001}. The construction is the same as in the case where \(X\) is smooth.

Dualising the inclusion
\(T_{\mathcal F}\hookrightarrow T_X\) gives a morphism
\[
    \phi:\Omega^1_X\longrightarrow T_{\mathcal F}^*
    \simeq \mathcal O_X(K_{\mathcal F}).
\]
Let
\(\mathcal U=\{U_\lambda\}_{\lambda\in\Lambda}\) be a trivialising open cover of \(X\)
for both \(\mathcal O_X(K_{\mathcal F})\) and \(L\).
For each \(\lambda\), let \(\alpha_\lambda\) be a generator of
\(L|_{U_\lambda}\) and let \(\beta_\lambda\) be a generator of
\(\mathcal O_X(K_{\mathcal F})|_{U_\lambda}\).
Let \(l_{\lambda\mu}\) and \(e_{\lambda\mu}\) be the transition functions:
\[
    \alpha_\mu=l_{\lambda\mu}\alpha_\lambda,
    \qquad
    \beta_\mu=e_{\lambda\mu}\beta_\lambda .
\]

Thus, if \(f\in\mathcal O_{U_\lambda}\), we define a local derivation
\(\partial_\lambda\) by
\[
    \phi(df)=\partial_\lambda(f)\beta_\lambda .
\]
On \(U_\lambda\cap U_\mu\), we then have
\[
    \partial_\lambda=e_{\lambda\mu}\partial_\mu .
\]

Let \(s\) be a local section of \(L\). On \(U_\lambda\), write
\[
    s=s_\lambda^{(0)}\alpha_\lambda,
\quad s_\lambda^{(0)}\in\mathcal O_{U_\lambda}.
\]
For $i\ge 1$, define 
\[
    s_\lambda^{(i)}:=\partial_\lambda(s_\lambda^{(i-1)}).
\]

A direct computation shows that on overlaps \(U_\lambda\cap U_\mu\),
the vectors
\[
    (s_\lambda^{(0)},\dots,s_\lambda^{(m)})
    \quad\text{and}\quad
    (s_\mu^{(0)},\dots,s_\mu^{(m)})
\]
are related by a lower-triangular $(m+1)\times (m+1)$ matrix $A_{\lambda \mu}^{(m)}$ whose diagonal entries are
\[
    l_{\lambda\mu},\; l_{\lambda\mu}e_{\lambda\mu},\;\dots,\; l_{\lambda\mu}e_{\lambda\mu}^m.
\]
Hence these transition functions define a rank \(m+1\) vector bundle
\(J^m_{\mathcal F}L\). Consider the map
 \[
     \pi_m:J^m_{\mathcal F}L\rightarrow J^{m-1}_{\mathcal F}L,
 \]
 which is locally given by
 \[
     \pi_m(s_\lambda^{(0)},s_\lambda^{(1)},\ldots,s_\lambda^{(m)})
     =
     (s_\lambda^{(0)},\ldots,s_\lambda^{(m-1)}).
 \]
 Since the transition matrices \(A_{\lambda \mu}^{(m)}\) of \(J^m_{\mathcal F}L\) are lower triangular, $\pi_m$ glues to give a surjective morphism of vector bundles. We claim that
 \[
     \ker(\pi_m)\cong \mathcal O_X(mK_{\mathcal F})\otimes L.
 \]
 Indeed, locally, we have
 \[
     \ker(\pi_m)|_{U_\lambda}
     =
     \{(0,\ldots,0,a_\lambda)\mid a_\lambda\in\mathcal O_{U_\lambda}\}.
 \]
 On $U_\lambda\cap U_\mu$, the transition matrix \(A_{\lambda \mu}^{(m)}\) identifies
 \[
     (0,\ldots,0,a_\mu)
 \]
 with
 \[
     (0,\ldots,0,l_{\lambda\mu}e_{\lambda\mu}^m a_\mu).
 \]
 Thus, $\ker(\pi_m)$ is locally free of rank one with transition function
 $l_{\lambda\mu}e_{\lambda\mu}^m$. Therefore
 \[
     \ker(\pi_m)\cong \mathcal O_X(mK_{\mathcal F})\otimes L,
 \]
 as claimed, and this gives the required exact sequence.

\medskip

We now construct the natural map
\[
j^m_{\mathcal F}: L \longrightarrow J^m_{\mathcal F}L.
\]

Let \(s\in H^0(V,L)\) for some open set \(V\subset X\).
On each \(U_\lambda\cap V\), write
\[
    s=f_\lambda \alpha_\lambda,
\quad f_\lambda\in\mathcal O_{U_\lambda\cap V}.
\]
Define the local jet vector
\[
    j^m_{\mathcal F,\lambda}(s)
    :=
    \bigl(f_\lambda,\partial_\lambda f_\lambda,\dots,\partial_\lambda^m f_\lambda\bigr).
\]
We first check compatibility with the transition functions of \(J^m_{\mathcal F}L\).
On overlaps \(U_\lambda\cap U_\mu\),
we have
\[
    f_\lambda = l_{\lambda\mu} f_\mu,
\quad\text{and}\quad
\partial_\lambda = e_{\lambda\mu}\partial_\mu.
\]

By proceeding by induction on \(r\ge0\) and by the Leibniz rule, we have that 
\[
\partial_\lambda^r(f_\lambda)
=
\sum_{q=0}^r a^{r,q}_{\lambda\mu}\,\partial_\mu^q(f_\mu),
\]
where \(a^{r,q}_{\lambda\mu}\in \mathcal O_X(U_\lambda\cap U_\mu)\)
are the coefficients in the transition matrices 
\(A_{\lambda \mu}^{(m)}\)
defining \(J^m_{\mathcal F}L\).
In particular, this gives
\[
    j^m_{\mathcal F,\lambda}(s)
    =
    A^{(m)}_{\lambda\mu}\, j^m_{\mathcal F,\mu}(s).
\]
Therefore the local vectors \(j^m_{\mathcal F,\lambda}(s)\) glue to a
global section
\[
    j^m_{\mathcal F}(s)\in H^0(V,J^m_{\mathcal F}L).
\]
This defines a \(\mathbb C\)-linear morphism
\(
    j^m_{\mathcal F}: L \longrightarrow J^m_{\mathcal F}L
\).
\end{proof}

\begin{lemma}\label{NC_Jets}
    Let $X$ be a normal projective variety of dimension $n$ 
    and let $\mathcal F$ be a foliation of rank one on $X$ such that 
    $K_{\mathcal F}$ is $\mathbb Q$-Cartier. 
    Let $L$ be an ample line bundle on $X$.
    Let 
    $j\colon U\hookrightarrow X$ be the inclusion of the open set $U$ where 
    $K_{\mathcal F}$ is Cartier. 
        We set $ J^0_L\coloneqq L$ and, for any $m\geq 0$, 
    \[
         J^m_L := j_*J^m_{\mathcal F_U}(L|_U),
        \qquad
         L_m := \mathcal O_X(mK_{\mathcal F})\otimes L.
    \]

    Then, after possibly replacing $L$ by a 
    sufficiently high multiple, we have that:
    \begin{enumerate}
        \item $ L_m$ is globally generated 
        for $m=0,1,\ldots,n$;

        \item $H^1(X, L_m)=0$ for $m=0,1,\ldots,n$;

        \item for every $m=1,\ldots,n$, if
        \[
            \psi_m: J^m_L\longrightarrow  J^{m-1}_L
        \]
        is the natural morphism and
        \(
             G_{m-1}:=\operatorname{Im}(\psi_m),
        \)
        then there is a short exact sequence
        \[
            0\longrightarrow 
             L_m
            \longrightarrow 
             J^m_L
            \longrightarrow 
             G_{m-1}
            \longrightarrow 0.
        \]
        Moreover,
        \[
             G_{m-1}|_U\simeq J^{m-1}_{\mathcal F_U}(L|_U),
        \]
        and the quotient
        \(
             J^{m-1}_L/ G_{m-1}
        \)
        is supported on $X\setminus U$. 

  \item a general section $s\in H^0(X,L)$ has the property that its 
foliated $n$-jet
\[
    j^n_{\mathcal F}(s)\in H^0(X,j_*J^n_{\mathcal F_U}(L|_U))
\]
is nowhere vanishing on
\[
    \Omega:=U\setminus \operatorname{Sing}^+\mathcal F .
\]
    \end{enumerate}
\end{lemma}

\begin{proof}
After possibly replacing $L$ by a sufficiently high multiple, Serre 
    vanishing and global generation give (1) and (2). 
    On $U$, Lemma \ref{Jets 1} gives the exact sequence
    \[
        0\longrightarrow 
        \mathcal O_U(mK_{\mathcal F})\otimes L|_U
        \longrightarrow 
        J^m_{\mathcal F_U}(L|_U)
        \longrightarrow 
        J^{m-1}_{\mathcal F_U}(L|_U)
        \longrightarrow 0.
    \]
    Applying $j_*$ gives a left exact sequence and a natural morphism
    \[
        \psi_m:
        j_*J^m_{\mathcal F_U}(L|_U)
        \longrightarrow
        j_*J^{m-1}_{\mathcal F_U}(L|_U).
    \]
    Define
    \(
         G_{m-1}:=\operatorname{Im}(\psi_m).
    \)
    Then we have a short exact sequence
    \[
        0\longrightarrow 
        j_*\big(\mathcal O_U(mK_{\mathcal F})\otimes L|_U\big)
        \longrightarrow 
         J^m_L
        \longrightarrow 
         G_{m-1}
        \longrightarrow 0.
    \]
    Since $X$ is normal, $\operatorname{codim}_X(X\setminus U)\geq 2$, and both sheaves are reflexive, we have
    \[
        j_*\big(\mathcal O_U(mK_{\mathcal F})\otimes L|_U\big)
        \simeq
        \mathcal O_X(mK_{\mathcal F})\otimes L
        =:
         L_m.
    \]
    Thus,
    \[
        0\longrightarrow 
         L_m
        \longrightarrow 
         J^m_L
        \longrightarrow 
         G_{m-1}
        \longrightarrow 0
    \]
    is exact.
    Since the original sequence is surjective on $U$, we also have
    \[
         G_{m-1}|_U\simeq J^{m-1}_{\mathcal F_U}(L|_U).
    \]
    Hence the quotient
    \(
         J^{m-1}_L/ G_{m-1}
    \)
    is supported on $X\setminus U$, so (3) holds.


Finally, we prove (4). After possibly replacing again $L$ by a sufficiently high multiple, we may assume that
$L$ separates ordinary $n$-jets at every point of $\Omega$, i.e. 
for every closed point $x\in \Omega$, the natural map
\[
    H^0(X,L)\longrightarrow L\otimes
    \mathcal O_X/\mathfrak m_x^{n+1}
\]
is surjective.

We claim that the foliated $n$-jets of global sections of $L$ generate
$J^n_{\mathcal F}L$ at every point of $\Omega$. Indeed, let $x\in \Omega$.
Since $\mathcal F$ is regular at $x$ and $K_{\mathcal F}$ is Cartier on
$U$, after shrinking around $x$ we may choose a local generator
$\partial$ of $\mathcal F$. Then, after trivialising $L$ locally, the
foliated $n$-jet of a section $s$ of $L$ at $x$ is given by
\[
    j^n_{\mathcal F}(s)(x)
    =
    \big(
        s(x),
        \partial s(x),
        \ldots,
        \partial^n s(x)
    \big).
\]
Note that the evaluation map
\[
    e_n(x)\colon H^0(X,L)\to J^n_{\mathcal F}L\otimes k(x),
    \qquad
    s\mapsto j^n_{\mathcal F}(s)(x)
\]
factors as 
\[
H^0(X,L)\to L\otimes \mathcal O_X/\mathfrak m_x^{n+1}\to 
J^n_{\mathcal F}L\otimes k(x),
\]
where the second map is surjective by construction and the fact that $x\notin {\rm Sing}^+\mathcal F$. 
Thus, since $L$ separates ordinary $n$-jets at $x$, we have that $e_n(x)$
is surjective and our claim follows.

Let $V:=H^0(X,L)$ and consider
the incidence correspondence
\[
    I:=
    \left\{
        (x,[s])\in \Omega\times \mathbb P(V)
        \ \middle|\
        j^n_{\mathcal F}(s)(x)=0
    \right\}.
\]
For fixed $x\in \Omega$, the condition
\(
    j^n_{\mathcal F}(s)(x)=0
\)
imposes $n+1$ independent linear conditions on $V$, because the evaluation
map above is surjective and the rank of 
\(
     J^n_{\mathcal F}L
\)
is $n+1$.
Hence
\[
    \dim I
    \leq
    \dim \Omega+\dim \mathbb P(V)-(n+1)
    <
    \dim \mathbb P(V).
\]
Therefore the image of $I$ in $\mathbb P(V)$ is a proper subset. It follows
that for a general section $s\in H^0(X,L)$, the  jet
\(
    j^n_{\mathcal F}(s)
\)
is nowhere vanishing on $\Omega$.
\end{proof}

\begin{theorem}\label{t_Bertini}
Let $X$ be a normal projective variety of dimension $n$ and let
$\mathcal F$ be a foliation of rank one on $X$ with
 canonical singularities.
Let $L$ be an ample line bundle on $X$. 

Then, after possibly replacing $L$ by a
multiple, we may find $A \in |L|$ such that $A$ does not contain any
log canonical  centre of $\mathcal F$ which is not contained in
$\operatorname{Sing}^+\mathcal F$.
\end{theorem}

\begin{proof}
Let $L$ be an ample line bundle on $X$ and replace $L$ by the multiple
guaranteed to exist by Lemma \ref{NC_Jets}. Let
\(
j\colon U\hookrightarrow X
\)
be the inclusion of the open set where $K_{\mathcal F}$ is Cartier, and set
\(
    \Omega:=U\setminus \operatorname{Sing}^+\mathcal F .
\)

By Lemma \ref{NC_Jets} (4), we may choose a
general section
\(
    s\in H^0(X,L)
\)
such that the jet
\[
    j^n_{\mathcal F}(s)
    \in
    H^0\bigl(X,j_*J^n_{\mathcal F_U}(L|_U)\bigr)
\]
is nowhere vanishing on $\Omega$. Let
\[
    A:=\{s=0\}\in |L|.
\]

We claim that $A$ does not contain any log canonical centre of $\mathcal F$ which
is not contained in $\operatorname{Sing}^+\mathcal F$. Suppose, by 
contradiction, that there exists such a log canonical centre $Z\subset A$.

First, we show that $Z$ is $\mathcal F$-invariant.
Let $z\in Z$ be a general point. There is an analytic open
neighbourhood $V_z$ of $z$ and a quasi-\'etale cover
\(
    \sigma\colon V'\to V_z
\)
such that the induced foliation $\mathcal F'$ on $V'$ has Gorenstein
singularities. Let $Z'$ be an irreducible component of
$\sigma^{-1}(Z\cap V_z)$ dominating $Z\cap V_z$. Then 
\cite[Lemma 2.20]{spicersvaldi} implies that $Z'$ is a log canonical
centre of $\mathcal F'$. Since
\[
    Z'\not\subset {\rm Sing}^+\mathcal F'
    =
    \sigma^{-1}({\rm Sing}^+\mathcal F),
\]
it follows from \cite[Lemma 2.9]{CS20} that $Z'$ is
$\mathcal F'$-invariant. Hence,  by \cite[Lemma 4.2]{CS23}, $Z$ is $\mathcal F$-invariant. In particular, $K_{\mathcal F}$
is Cartier at the general point of $Z$, i.e.  the generic
point of $Z$ lies in $U$.

\medskip 

We now work locally near a general point of $Z$. Choose an affine open
subset
\(
    W\subset \Omega
\)
such that $W\cap Z\neq \emptyset$, $L|_W$ is trivial, and
$T_{\mathcal F}|_W$ is generated by a vector field $\partial$. Let
\(
    \mathcal I:=\mathcal I_{Z\cap W}
\)
be the ideal sheaf of $Z\cap W$ in $W$. After choosing a trivialisation
of $L|_W$, write
\[
    s|_W=\widetilde{s}\in \mathcal O_W .
\]
Since $Z\subset A$, we have
\(
    \widetilde{s}\in \mathcal I.
\)
Since $Z$ is $\mathcal F$-invariant, we have
\(
    \partial(\mathcal I)\subset \mathcal I.
\)
Therefore, by induction,
\[
    \partial^i(\widetilde{s})\in \mathcal I
    \qquad
    \text{for all } i=0,1,\ldots,n.
\]

Under the induced local trivialisation of $J^n_{\mathcal F}L$ on $W$,
the foliated $n$-th jet of $s$ is
\[
    j^n_{\mathcal F}(s)|_W
    =
    \bigl(
        \widetilde{s},
        \partial(\widetilde{s}),
        \ldots,
        \partial^n(\widetilde{s})
    \bigr).
\]
All of these components lie in $\mathcal I$. Hence
\(
    j^n_{\mathcal F}(s)|_W
\)
vanishes along $Z\cap W$.

This contradicts the choice of $s$, since
$j^n_{\mathcal F}(s)$ is nowhere vanishing on $\Omega$. Therefore $A$ does not contain
any log canonical centre of $\mathcal F$ which is not contained in
${\rm Sing}^+\mathcal F$.
\end{proof}

\begin{proof}[Proof of Theorem \ref{t_uniqueness_min_models}]
By assumption, there exists a birational map $\phi\colon Y_1 \dashrightarrow Y_2$ such that $\phi \circ f_1 = f_2$.

We first show that $\phi$ is an isomorphism in codimension one. Let $g_1\colon Z\rightarrow Y_1$ and $g_2\colon Z\rightarrow Y_2$ be birational morphisms which resolve the indeterminacy locus 
of $\phi$. As $K_{\mathcal{F}_1}$ and $K_{\mathcal{F}_2}$ are both nef, the negativity lemma implies that $g_1^*K_{\mathcal{F}_1}=g_2^*K_{\mathcal{F}_2}$. 
Since $f_1$ and $f_2$ are $K_{\cal F}$-negative, it follows that $\exc f_1= \exc f_2$. Thus, $\phi$ is small (see also  \cite[Lemma 3.1]{JV23}). 
Let us assume by contradiction that $\phi$ is not an isomorphism.

\medskip 
We claim that  $\exc \phi$ does not contain any log canonical centre of $\cal F_1$. Indeed, 
let $V\subset Y_1$ be a log canonical centre of $\mathcal{F}_1$. After possibly replacing $Z$ by a higher model, we may assume that $g_1$ extracts an exceptional divisor $E$ centred on $V$ and such that $a(E,\cal F_1)=0$. If $V_X$ is the image of $E$ in 
$X$, then as $f_1$ is $K_{\cal F}$-negative and $\cal F$ is canonical, it follows that $a(E,\cal F)=0$. Thus, the negativity lemma (cf. \cite[Lemma 2.7]{CS21}) implies that $f_1$ is an  isomorphism at the generic point of  $V_X$. 
Since $\phi$ is small, we have that $E$ is  $g_2$-exceptional and since  $g_1^*K_{\mathcal{F}_1}=g_2^*K_{\mathcal{F}_2}$, we have that  $g_2(E)\subset Y_2$ is  a log canonical  centre for $\mathcal{F}_2$ and, similarly as above, we have that $f_2$ is an isomorphism at the generic point of $V_X$. 
Hence the birational map $\phi\colon Y_1\dashrightarrow Y_2$ is an isomorphism at the generic point  of $V$ and, therefore, our claim follows. 
Similarly, $\exc \phi^{-1}$ does not contain any log canonical centre of $\mathcal F_2$.
 	
\medskip

We now consider sufficiently general ample divisors $A_1$ and $A_2$ on $Y_1$ and $Y_2$ respectively, with  strict transforms $H_1$ on $Y_2$ and $H_2$ on $Y_1$ respectively. In particular, $A_1$ does not contain any $\cal F_1$-invariant subvariety which is not contained in ${\rm Sing}^+\cal{F}_1$ as in Theorem \ref{t_Bertini}, which implies that $H_1$ is not $\cal F_2$-invariant. Similarly, we can choose $A_2$ in this way.  Assume by contradiction that  $K_{\mathcal{F}_1}+ H_2$ and 
 $K_{\mathcal{F}_2}+ H_1$ are both nef, then $K_{\mathcal{F}_1}+ A_1+H_2$ and $K_{\mathcal{F}_2}+ A_2+ H_1$ are both ample divisors. Note that $\phi_*(K_{\mathcal{F}_1}+A_1+H_2)=K_{\mathcal{F}_2}+A_2+H_1$. We may write
				 \[
				 g_1^*(K_{\mathcal{F}_1}+ H_2+ A_1)-g_2^*(K_{\mathcal{F}_2}+ A_2+  H_1)=E_1-E_2,
				 \]
where $E_1$ and $E_2$ are effective $g_1$-exceptional divisors,  and hence also $g_2$-exceptional. The negativity lemma implies that $E_1=E_2$. In particular, $D:=g_1^*(K_{\mathcal{F}_1}+ H_2+ A_1)=g_2^*(K_{\mathcal{F}_2}+ A_2+  H_1)$ is a semi-ample divisor on $Z$. As $g_{1*}D$ and $g_{2*}D$ are both ample, it follows that $g_1$ and $g_2$  contract the same curves and, by the rigidity lemma, we have that  $\phi\colon Y_1\to Y_2$ is an isomorphism, contradicting our assumption. Thus, after possibly replacing $Y_1$ by $Y_2$ and $H_2$ by $H_1$, we may assume that $K_{\mathcal{F}_1}+ tH_2$ is not nef for any $0<t\ll 1$.

We now show that there exists $\epsilon >0$ such that $(\cal F_1,\epsilon H_2)$ is log canonical at every closed point of  $\exc \phi$. Let $p\in \exc \phi$ be a closed point. 
Since $\exc \phi$ does not contain any log canonical centre of $\mathcal{F}_1$ and  $\cal F_1$ is canonical, 
it follows that 
$\cal F_1$ is terminal at $p$.
Thus, by \cite[Lemma 2.9]{CS20}, there exists  an analytic open neighbourhood $V$ of $p$ and a quasi-\'etale cover $\pi\colon U\rightarrow V$ such that $\pi^{-1}\mathcal{F}_1|_U$ is induced by a holomorphic submersion $g\colon U\rightarrow W$, i.e. $\pi^{-1}\mathcal{F}_1|_U=T_{U/W}$. Let $\mathcal F'_1\coloneqq  \pi^{-1}\mathcal{F}_1$ and let 
$H'_2\coloneqq \pi^*H_2$. Then, \cite[Lemma 2.8]{CS20} implies that, for any closed point $q\in U$ and for any  $t>0$, we have that  $(\cal F_1, tH_2)$ is log canonical at $\pi(q)$ if and only if 
$(\mathcal F'_1,tH'_2)$ is log canonical at $q$. 
Let $L$ be a fibre of $g$ and assume by contradiction that $L$ is contained in the support of $H'_2$. Let $M$ be the Zariski closure of $\pi(L)$ in $Y_1$. Then $M$ is $\cal F_1$-invariant and it is contained in the support of $H_2$. In particular, by our choice of $A_2$, it follows that $M$ is contained in $\exc \phi$, a contradiction.
By \cite[Lemma 4.4]{CS23} we have that $(L,t H'_2|_L)$ is log canonical if and only if $(\mathcal{F}'_1, tH'_2)$ is log canonical in a neighbourhood of $L$ and our claim easily follows.

\medskip

Denote $D_1:= \epsilon H_2$. Then $K_{\mathcal{F}_1}+D_1$ is not nef and by the cone theorem (cf.  \cite[Theorem 4.8]{CS23}), there exists a curve 
$C$ on $Y_1$ such that $(K_{\cal F_1}+D_1)\cdot C<0$ and either $(\cal F_1,D_1)$ is not log canonical at the general point of  $C$ or $C$ is $\cal F_1$-invariant. 
Since $K_{\cal F_1}$ is nef, we have that $H_2\cdot C<0$. 
By the negativity lemma we may write $g_1^*H_2=g_2^*A_2+E$, where $E\ge 0$ is $g_2$-exceptional. As $Y_1$ and $Y_2$ are isomorphic in codimension one, we have that $E$ is also $g_1$-exceptional.  Moreover, by construction, we have that $g_1(\supp{E})$ is contained in $\exc \phi$.
Let $C'$ be 
a curve on $Z$ which maps onto $C$. Then $(g_2^*A_2+E)\cdot C'=g_1^*H_2\cdot C'<0$. Since $A_2$ is ample, it follows that $E\cdot C'<0$, 
which implies that $C$ is contained in $\exc \phi$. Thus, $(\cal F_1,D_1)$ is log canonical at the general point of  $C$ and therefore $C$ is $\cal F_1$-invariant. In particular, $C$ is a log canonical centre of $\cal F_1$ which is contained in $\exc \phi$, contradicting our claim above.
\end{proof}

\begin{remark}
	In Theorem \ref{t_uniqueness_min_models} it is essential that $Y_1$ and $Y_2$ are outputs of $K_{\mathcal{F}}$-MMPs starting from the same variety, i.e. it is not true that rank one foliated minimal models are unique in a birational class. For example, let $E$ be an elliptic curve, consider $Y_1:=\mathbb{P}^2\times E$ and the rank one foliation $\mathcal{F}_1$ defined by the morphism $\mathbb{P}^2\times E\rightarrow \mathbb{P}^2$. Now, consider $Y_2:=\mathbb{P}^1\times \mathbb{P}^1\times E$ and the rank one foliation  $\mathcal{F}_2$ defined by the morphism $Y_2\rightarrow \mathbb P^1\times \mathbb P^1$. Note that both $K_{\mathcal{F}_1}$ and $K_{\mathcal{F}_2}$ are nef and $Y_1$ and $Y_2$ are birational algebraic varieties, however, $Y_1$ and $Y_2$ are not isomorphic.

\end{remark}

\section{Existence of F-dlt flops for co-rank one foliations}
\label{s_co-rank_one_flops}
The goal of this section is to prove the existence of flops in several cases.

\subsection{F-dlt flops}

We begin with the proof of our second main theorem. 

\begin{proof}[Proof of Theorem \ref{thm_dlt_flop}]
We first observe that if $Z' \to Z$ is any \'etale neighbourhood of $f(C)$ then the pull-back of $\mathcal F$ to $X\times_ZZ'$ is also F-dlt.  Since the construction of the $D$-ample model is \'etale local on the base we may therefore freely replace $Z$ by any \'etale neighbourhood of $P\coloneqq f(C)$.

Let $\hat{Z}$ be the formal completion of $Z$ along $P$ so we have a projective morphism of formal schemes
$\hat{f}\colon \hat{X} \to \hat{Z}$.

\medskip

We now turn to the key point in this proof.  We will show there exists $j \in \{1, \dots, N\}$ such that 
the separatrix $S_j$ meets $C$ properly, i.e., $S_j \cap C \neq \emptyset$ and $C \not\subset S_j$, and in particular $S_j\cdot C>0$.
We first prove the following:

\begin{claim}
\label{claim_2}
We have that  $\sing \mathcal F \cap C \neq \emptyset$.
Moreover,  if there exists a component of $\sing \mathcal F \cap S_1$ distinct from $C$, then there exists a separatrix $S_j$
    meeting $C$ properly. 
\end{claim}  
\begin{proof}[Proof of Claim]
    Suppose for sake of contradiction that $\sing \mathcal F \cap C = \emptyset$.
    Then there is at most one separatrix meeting $C$. Indeed, if there were two separatrices, $S$ and $S'$, at some point $x \in C$, then $x \in S \cap S' \subset \sing \mathcal F$.  Thus, $S_1$ is the unique separatrix meeting $C$.  By the Camacho-Sad formula (cf. \cite[Theorem 3.9]{P22}) it follows that $S_1|_{S_1} \equiv 0$, and therefore $S_1$ is $f$-numerically trivial, contrary to hypothesis.

    \medskip

    To verify the second part of the claim suppose that $\Sigma$ is  a component of $\sing \mathcal F \cap S_1$ which meets $C$ at a point $x \in C$ and $\Sigma \neq C$.  Since $\mathcal F$ has F-dlt singularities, it has simple singularities at a general point of $\Sigma$ and so we may find two (local) separatrices $S$ and $S'$ at a general point of $\Sigma$. 
    Since $C \cup \Sigma$ is tangent to $\mathcal F$ we may extend $S$ and $S'$ to a (formal) neighbourhood of $C$ and we may assume that $S \cap C \neq \emptyset$ and $S' \cap C \neq \emptyset$.
    
    One of these separatrices, say $S$, must be equal to $S_1$, and since $S_1$ is normal $S'$ is distinct from $S_1$, and so up to relabelling we may assume that $S' = S_2$.   We claim that $S_2$ does not contain $C$.

Assume by contradiction that $C\subset S_2$. We first show that $\mathcal F$ has simple singularities in a neighbourhood of $x \in X$.  Indeed, the point $x$
is an lc centre of $(X, \sum S_k)$
and hence \cite[Lemma 8.14]{spicer20} (and its proof) imply that $x$ is an lc centre of $\mathcal F$.  Since $\mathcal F$ is F-dlt, \cite[Lemma 3.8]{CS21} implies that $\mathcal F$ has simple singularities in a neighbourhood of $x$ (and in particular, $x \in X$ is a smooth point).  It follows that in a  neighbourhood of $x \in X$ we may find (formal) coordinates $z_1, z_2, z_3$ such that $\{z_1 = 0\}\cup \{z_2 = 0\}\cup \{z_3 = 0\}$ are the (local) separatrices of $\mathcal F$
at $x \in X$.  Since $S_1$ is normal, up to relabelling we may assume that $S_1 = \{z_1 = 0\}$.  
Therefore, near \(x\), either \(C\) is smooth and is given by
\(\{z_1=z_2=0\}\), or \(C\) has two formal branches at \(x\) and is
given by \(\{z_1=z_2z_3=0\}\). In the former case, $\{z_3 = 0\}$ is a local separatrix which, when extended to a neighbourhood of $C$, meets $C$ properly.  In the latter case, we see that in fact $C = \sing \mathcal F \cap S_1$ near $x$, contrary to our hypothesis that $\Sigma$ was distinct from $C$.
\end{proof}

Suppose for sake of contradiction that there does not exist a separatrix $S_j$ meeting $C$ properly. 
By Claim \ref{claim_2}
we see that it must be the case that $\sing \mathcal F = C$ in a neighbourhood of $C$ and that for any 
   $j \neq 1$, we have  $S_1 \cap S_j \subset C$.
   Since $\mathcal F$ has F-dlt singularities, it has simple singularities at a general point of $C$ and so, it follows that there are at most two separatrices at a general point of $C$ and thus we see that the set of separatrices meeting $C$ consists of exactly two separatrices $S_1$ and $S_2$.

 There are two possibilities for the transverse type of the foliation singularity along $C$ (cf. \cite[pp. 2-3]{brunella00}).  Either
	\begin{itemize}
		\item the transverse type of the foliation singularity along $C$ is logarithmic; or 
\item the transverse type of the singularity along $C$
	is of saddle node type.
	\end{itemize}

Suppose for the moment  that the transverse type of the foliation singularity along $C$ is of saddle node type.  We claim in this case that $S_1$ is necessarily the strong separatrix.  Suppose for sake of contradiction that $S_1$ is the weak (formal) separatrix of the foliation singularity along $C$. 
In this case, by foliation adjunction we may write
	\[K_{\mathcal F}\vert_{S_1} = K_{S_1}+\alpha C\] where $\alpha > 1$ depends only on the multiplicity of the saddle node.  Since $f$ is a flopping contraction 
	we see that 
    \begin{align}
    \label{flop_eq}
    (K_{S_1}+\alpha C)\cdot C = 0.
    \end{align}
Let $S\coloneqq S_1+S_2$. By Lemma \ref{lem_nef_comparison_lem},
we have that 
\[
(K_X+S)\cdot C \leq K_{\mathcal F}\cdot C = 0.
\]
	Since \[(K_X+S)\vert_{S_1} = K_{S_1}+C\] 
	we deduce that 
    \begin{align}
    \label{flop_ineq}
        (K_{S_1}+C)\cdot C \leq 0.
    \end{align}
By taking the difference of \eqref{flop_eq} and \eqref{flop_ineq} we see that 
	\[(\alpha-1)C\cdot C \geq 0\]
 which contradicts the fact that $C \subset S_1$ is contractible. This is our sought after contradiction, and so $S_1$ must be the strong separatrix.

    \medskip

So we may assume that either the transverse type of the foliation singularity at $C$ is of logarithmic type, or, that the transverse type of the foliation is saddle node type and $S_1$ is the strong separatrix.
By applying
	the Camacho-Sad formula 
	 we may write \[S_1\vert_{S_1} \equiv aC\] with 
	$a = {\rm CS}(\mathcal F\vert_H, S_1\cap H, C\cap H)$ where $H$ is a general hyperplane section passing through a general point of $C$
    (cf. \cite[Theorem 3.9]{P22}).  We note that by \cite[pp. 30-31]{brunella00} $a \notin \mathbb Q_{>0}$.

On the other hand, 	since $C$ is contractible and since $-S_1\vert_{S_1}$ is a $\mathbb Q$-Cartier and $f$-ample divisor, the negativity lemma, \cite[Lemma 3.39]{km98}, implies that $a \in \mathbb Q_{>0}$.
This is our sought after contradiction.

\medskip

We have therefore found a separatrix which meets $C$ properly.  Up to relabelling we may assume that this separatrix is $S_2$.
By \cite[Lemma 8.14]{spicer20} (see also the paragraph after \cite[Lemma 3.16]{CS21}) the pair $(\hat{X}, S\coloneqq  \sum_{i = 1}^N S_i)$ is log canonical and by Lemma \ref{lem_nef_comparison_lem}, we have that  $-(K_{\hat X}+S)$ is $f$-nef.
 Since $S_2 \cdot C>0$ it follows that for  $0<\epsilon \ll \eta \ll1 $, the pair $(\hat{X}, \Delta \coloneqq (1-\epsilon)(S - \eta S_2))$ is klt and $-(K_{\hat X}+\Delta)$ is $f$-ample.

Following \cite[\S 5]{CS21}, up to replacing $Z$ by an \'etale neighbourhood we may approximate $\Delta$ by a divisor $\tilde{\Delta}$ on $X$ such that $(X, \tilde{\Delta})$ is klt and $-(K_X+\tilde{\Delta})$ is $f$-ample.
We may then apply \cite[Corollary 1.3.1]{BCHM06} to produce the $D$-ample model, as required.
\end{proof}

\begin{remark}
\label{rem_cs_singular}
    In fact, the above argument works in slightly greater generality.  It suffices to assume that $X$ is analytically $\mathbb Q$-factorial (hence $S_1, \dots, S_N$ are $\mathbb Q$-Cartier).  The Camacho-Sad formula used above (\cite[Theorem 3.9]{P22}) is stated for smooth complex varieties but the proof easily adapts to the setting where $X$ has quotient singularities in codimension two, cf. \cite[the proof of Proposition 3.12]{DO19}.
\end{remark}


\begin{remark}
\label{dlt_flop_nonexist}
We remark that Theorem \ref{thm_dlt_flop} is perhaps surprising. Flops
for dlt pairs do not always exist (cf. \cite[Example 7.7]{Fujino15}).
In Fujino's example, the flopping curve has genus one, and a certain
non-torsion line bundle on the curve obstructs the finite generation of
the relevant algebra. The proof of Theorem \ref{thm_dlt_flop} shows that (under the assumption of the Theorem)
this phenomenon cannot occur for co-rank one foliations. Indeed, in this case the flopping curve is always a smooth
rational curve.
\end{remark}

\subsection{Producing klt flops}

\begin{proof}[Proof of Theorem \ref{thm_flop_klt}]
We first show that $\exc f$ is tangent to $\mathcal F$. Let $C$ be an irreducible component of $\exc f$.
Suppose for sake of contradiction that $C$ is not tangent to $\mathcal F$.
Since $C\subset \exc f$ and since $\mathcal F$ is log terminal at the generic point of $C$,  there exists a $\mathbb Q$-Cartier $\mathbb Q$-divisor $G \geq 0$ such that $G\cdot C<0$ and such that $C$ is a log
canonical centre of $(\mathcal F, G)$.  On one hand, we have $(K_{\mathcal F}+G)\cdot C<0$. On the other hand, since $C$ is not tangent to $\mathcal F$, by
\cite[Theorem 4.5]{spicer20}, we have that $(K_{\mathcal F}+G)\cdot C \geq 0$, a contradiction.

We now claim that we may freely replace $Z$ by an \'etale neighbourhood of $f(\exc f)$ and replace $X$ by a small $\mathbb Q$-factorialisation, so that if $S_1, \dots, S_N \subset \hat{X}$ are all the separatrices meeting $\exc f$ (where $\hat{X}$ is the formal completion of $X$ along $\exc f$), then $S_i$ is $\mathbb Q$-Cartier.

By the arguments in \cite[\S 4 and \S 5]{CS21}, up to replacing $Z$ by an \'etale neighbourhood of $P\coloneqq f(\exc f)$ we may assume that there exist $m>0$ and rank one reflexive sheaves $M_1, \dots, M_N$ on $X$ so that $M_i|_{\hat{X}} \cong \mathcal O_{\hat{X}}(mS_i)$.

Let $\mu\colon Y \to X$ be a small $\mathbb Q$-factorialisation of $X$ which is guaranteed to exist by \cite[Corollary 1.4.3]{BCHM06}.
Let $\tilde{Z} := \Spec \widehat{\mathcal O_{Z, P}}$, let $\tilde{X} := X\times_Z\tilde{Z}$ and let $\tilde{Y} := Y\times_Z\tilde{Z}$.
By the Grothendieck existence theorem $S_i$ corresponds to a divisor on $\tilde{X}$, which (by abuse of notation) we will continue to denote $S_i$ and which 
continues to satisfy $M_i|_{\tilde{X}} \cong \mathcal O_{\tilde{X}}(mS_i)$.  Let $M^Y_i$ (resp. $S^Y_i$) denote the strict transform of $M_i$ (resp. $S_i$) on $Y$ (resp. on $\tilde{Y}$).
Since $\tilde{Y}$ and $\tilde{X}$ are isomorphic away from a set of codimension at least $2$, and since reflexive sheaves are uniquely determined away from subsets of codimension at least $2$, 
we deduce that $M^Y_i|_{\tilde{Y}} \cong \mathcal O_{\tilde{Y}}(mS^Y_i)$.  Since $Y$ is $\mathbb Q$-factorial we see that $M^Y_i$ is $\mathbb Q$-Cartier, and hence
$S^Y_i$ and $S^Y_i|_{\hat{Y}}$ are $\mathbb Q$-Cartier where $\hat{Y}$ is the formal completion of $Y$ along $\exc{(f\circ\mu)}$.

Finally, we claim that any separatrix (formal or otherwise) of $\mu^{-1}\mathcal F$ meeting $\exc{(f\circ\mu)}$ is in fact equal to $S_i^Y$ for some $i \in \{1, \dots, N\}$.
Let us denote by $\hat{\mu}$ the morphism $\hat{Y} \to \hat{X}$ induced by $\mu$.
We claim that if $T$ is a reduced irreducible codimension-one subscheme of $\hat{Y}$ which is invariant under $\mu^{-1}\mathcal F|_{\hat{Y}}$ then $\hat{\mu}_*T$ is a codimension one reduced and irreducible subscheme of $\hat{X}$ which is again invariant under $\mathcal F|_{\hat{X}}$.  To prove this latter claim, first note that invariance of a divisor may be checked away from a closed subset not equal to the divisor.  So,
by considering $\hat{\mu}_*T$ as a divisor on $\tilde{X}$, and noting that $\mu$ is an isomorphism away from $\exc \mu$,  it follows that 
$\hat{\mu}_*T$ is invariant away from $\mu(\exc \mu)$ and hence $\hat{\mu}_*T$ is invariant.   Thus, $\hat{\mu}_*T$ is a separatrix meeting $\exc f$ and hence it is equal to $S_i$ for some $i \in \{1, \dots, N\}$ and hence $T = S_i^Y$.

So up to replacing $X$ by $Y$ we may assume that $S_i$ is $\mathbb Q$-Cartier.
Set $S:= \sum_{i = 1}^N S_i$.

Since $(\mathcal F, \Delta)$ is klt we see that there are no one dimensional components of $\sing \mathcal F$ (see for instance the discrepancy calculation in the proof of \cite[Lemma 3.3]{CS21}),  and so by the Camacho-Sad formula \cite[Theorem 3.9]{P22} (cf. also Remark \ref{rem_cs_singular})  we deduce 
that for any $i$, we have that $S_i|_{S_i} \equiv_{f} 0$.  From this we deduce that $S$ is $f$-nef.

By \cite[Lemma 8.14]{spicer20} (see also the paragraph after \cite[Lemma 3.16]{CS21})  $(\hat{X}, \Delta+S)$ is log terminal and so $(\hat{X}, \Gamma \coloneqq \Delta+S-\epsilon S)$ is klt for any $0<\epsilon <1$.  Lemma \ref{lem_nef_comparison_lem} implies that $-(K_{\hat{X}}+\Delta+S)$ is $f$-nef. Since $S$ is $f$-nef we then see that $-(K_{\hat{X}}+\Delta+S-\epsilon S)$ is $f$-nef.  As in the last paragraph of the proof of Theorem \ref{thm_dlt_flop}, by approximating $\Gamma$, and perhaps perturbing it slightly, we may produce a divisor $\tilde{\Gamma}$
so that 
$(X, \tilde{\Gamma})$ is klt and $-(K_X+\tilde{\Gamma})$ is $f$-ample.  We apply \cite[Corollary 1.3.1]{BCHM06} to produce the $D$-ample model, which is our required flop.
\end{proof}

\section{Examples}

The goal of this section is to collect some examples which explore the failure of the basepoint free
theorem for rank one foliations.
As shown in \cite{mcq08}, if $X$ is a surface with a rank one foliation $\cal F$ with canonical singularities
such that $K_{\cal F}$ is big and nef, it is not necessarily the case that $K_{\cal F}$ is semi-ample, the obstruction being given
by the presence of elliptic Gorenstein leaves (e.g.l.'s) contained in $\Null K_{\cal F}$.
We produce examples of similar behaviour in higher dimensions.

An important new feature arises in higher dimensions with regard to this phenomenon: on surfaces 
it is always possible to contract e.g.l.'s in the category of algebraic spaces, however 
if $X$ is a projective threefold and $\cal F$ is a rank one foliation on $X$ with canonical singularities, 
it may not be possible to contract $\Null K_{\cal F}$, even in the category of algebraic spaces.

\subsection{Hilbert modular surfaces}
\label{s_setup}
Let $X$ be a smooth Hilbert modular surface, obtained as the resolution of the cusps of a Baily-Borel compactification of a bi-disc quotient and equipped with
two  tautological foliations $\cal F$ and $\cal G$ (e.g. see \cite{mcq08} and \cite[Example 9.4]{brunella00}).
Let $E$ denote the e.g.l. in $X$, i.e. the pre-image of the cusp.
Recall that $E$ is either a cycle
of rational curves  or a nodal rational curve. We will assume below that $E$ is a nodal rational curve. 

We have:
\begin{enumerate}
\item $\cal F$ and $\cal G$ have canonical singularities; 
\item $K_{\cal F}$ and $K_{\cal G}$ are nef divisors such that 
\[H^0(X,\cal O_X(mK_{\cal F}))=H^0(X,\cal O_X(mK_{\cal G}))=0
\]
for all positive integers $m$;
\item $E$ is $\mathcal F$-invariant and $\cal G$-invariant;

\item $K_{\cal F}\cdot E = K_{\cal G}\cdot E= 0$; and

\item $K_{\cal F}\vert_E$ and $K_{\cal G}\vert_E$ are not torsion.
\end{enumerate}
Moreover, we have that $h^1(X,\cal O_X)=0$ and 
\[
\Omega^1_X(\log E) \simeq \cal O_X(K_{\cal F}) \oplus \cal O_X(K_{\cal G})
\]
and therefore 
\[K_{\cal F}+K_{\cal G}\sim K_X+E.
\]
We collect some easy calculations:

\begin{lemma}
\label{lem_hmf_calcs}
Set-up as above.

Then, for every positive integer $m$, we have
\[H^0(X, \Omega^1_X(\log E) \otimes \cal O_X(mK_{\cal F})) \neq 0.\]

\end{lemma}
\begin{proof}
Since $H^0(X, \cal O_X((m+1)K_{\cal F})) = 0$ and
$K_{\cal F}+K_{\cal G} \sim K_X+E$, we have that 
\[H^0(X, \Omega^1_X(\log E) \otimes \cal O_X(mK_{\cal F})) \simeq H^0(X, \cal O_X(K_X+E+(m-1)K_{\cal F})).\]

By Serre duality and since $K_{\cal F}$ is pseudo-effective and not numerically trivial, we have that 
\[
h^2(X, \cal O_X(K_X+E+(m-1)K_{\cal F})) = h^0(X, \cal O_X(-E-(m-1)K_{\cal F})) = 0.
\]
Since $h^{1}(X, \cal O_X) = 0$, we have that $\chi(X, \cal O_X) >0$.
By Riemann-Roch, we then see that 
\begin{align*}
h^0(X, \cal O_X(K_X+E+(m-1)K_{\cal F})) \geq \chi(X, \cal O_X(K_X+E+(m-1)K_{\cal F}))  \\
=\chi(X, \cal O_X) + \frac{(K_X+E+(m-1)K_{\cal F})\cdot (E+(m-1)K_{\cal F})}{2} > 0
\end{align*}
where the second inequality holds because $K_{\cal F}$ is nef, $(K_X+E)\cdot E=0$ 
and 
\[
K_{\cal F}\cdot (K_X+E) = K_{\cal F}\cdot (K_{\cal F}+K_{\cal G}) > 0.
\]
Thus, our claim follows. 
\end{proof}

\subsection{A rank one foliation which does not admit a flop}
\label{s_unfloppable}
We use the same set-up as in \S \ref{s_setup}. We denote by $\cal N_{\cal F}\coloneqq(T_X/T_{\cal F})^{**}$ the normal sheaf associated to $\cal F$. Since $(K_X+E)\cdot E=0$ and $E^2<0$, it follows that $K_X\cdot E>0$ and therefore $\cal N^*_{\cal F}\vert_E$ is ample.  Let $p\colon Y \coloneqq \bb V(\cal N_{\cal F}) \rightarrow X$ be the total space of the line bundle associated to $\cal N_{\cal F}$
and identify $X$ with the zero section $S \hookrightarrow Y$, so that the normal sheaf $\cal N_{S/Y}$ of $S$ in $Y$ coincides with $\cal N_{\cal F}$.
We will continue to denote by $E$ its image in the zero section.
Since $E^2<0$ in $X$ and $E\cdot S<0$, it follows that the normal sheaf $\cal N_{E/Y}$ is anti-ample and so, by \cite{artin70}, there exists a birational morphism $f\colon Y \rightarrow Z$ onto an algebraic space $Z$  which contracts $E$ to a point. 

\medskip 

We now want to construct a foliation $\cal H$ on $Y$ such that $S$ is $\cal H$-invariant and the restriction of $\cal H$ to $S$ coincides with $\cal F$. Let $X_0\coloneqq X\setminus \sing \cal F$. It is enough to construct an extension of $\cal F$ on $p^{-1}(X_0)$ and then consider its extension on $Y$. 
Let $\{U_i\}$ be a Zariski open cover of $X_0$ so that $\cal F$ is defined by a   $1$-form $\omega_i$ on $U_i$. Note that $\{\omega_i\}$ defines a non-zero element of $H^0(X, \Omega_X^1\otimes \cal N_{\cal F})$.
We have $p^{-1}(U_i)  \simeq U_i \times \bb C$ and we denote by $z_i$ the corresponding coordinate on $\bb C$.
Let $h_{ij}$ be the transition function of the line bundle associated to $\cal N_{\cal F}$, with respect to the open cover $\{U_i\}$, so that $\omega_{i} = h_{ij}\omega_j$  and  $z_i = h_{ij}z_j$, for all $i,j$. 

We want to define a collection of  $2$-forms $\Omega_i$ on $p^{-1}(U_i)$  
such that $\Omega_i = h_{ij}^2\Omega_j$, for all $i,j$  and which define a rank one foliation on $p^{-1}(X_0)$.

 On $p^{-1}(U_i)\simeq U_i\times \mathbb C$, we define 
\[\Omega_i \coloneqq z_i^2d\left(\frac{\omega_i}{z_i}\right)\]

It follows by a direct computation that $\Omega_i=h_{ij}^2\Omega_j$ and, since the $2$-forms $\Omega_i$ are locally decomposable, they define a rank one foliation on $p^{-1}(X_0)$, which extends to a foliation $\cal H$ on $Y$.
It follows easily that:

\begin{enumerate}
\item $K_{\cal H} = p^*K_{\cal F}$;

\item $S$ is $\cal H$-invariant and the restriction of $\cal H$  to $S$  (cf. \cite[Proposition-Definition 3.7]{CS23}) coincides with $\cal F$;

\item $E$ is $\cal H$-invariant; and

\item $K_{\cal H}\cdot E = 0$ but $K_{\cal H}\vert_E$ is not torsion.
\end{enumerate}

Observe that $\cal H$ restricted to $S$ has canonical singularities, and so by inversion of adjunction (cf. \cite[Theorem 4.12]{SSV25})
and \cite[Fact III.i.3]{MP13}
it follows that $\cal H$ has canonical singularities.

Perhaps replacing $Z$ by an affine \'etale neighbourhood of $f(E)$ (and recompactifying after a finite cover
if desired) we have produced an example of a flopping contraction where $K_{\cal H}$ is not torsion
on the flopping curve.

\medskip

We now consider an example of a $D$-flopping situation.
We denote 
\[
M \coloneqq \cal O_X(-E)\vert_E \quad \text{and} \quad L \coloneqq \cal O_X(K_{\cal G})\vert_E.
\]
Note that $M$ is ample and $L$ is of degree zero but not torsion. 

Let $\widehat{Y}$ denote the formal completion of $Y$ along $E$, $\widehat X$ the formal completion of $X$ along $E$ and  $\widehat{Z}$  the formal completion
of $Z$ along $f(E)$. Let $\widehat f\colon \widehat Y\to \widehat Z$ be the induced morphism.

Let $D \coloneqq p^*E$.
The existence of the $D$-flop is equivalent to the finite generation of the  $\cal O_{\widehat Z}$-algebra
\[
\cal A \coloneqq \bigoplus_{m\ge 0} \widehat f_*\cal O_{\widehat{Y}}(mD).
\]

Since $\cal N^*_{\cal F}(E) \simeq \cal O_X(K_{\cal G})$, we have that $\cal N_{\cal F}\vert_E \simeq M^* \otimes L^*$. Thus, 
\[
\begin{aligned}
\cal N_{E/Y} &\simeq \cal O_Y(D)\vert_E \oplus \cal O_Y(S)\vert_E \\
&\simeq \cal O_X(E)\vert_E \oplus \cal N_{\cal F}\vert_E\\
&\simeq   M^*\oplus (M^*\otimes L^*).
\end{aligned}
\]

Now let $Y_n$ denote the $n$-th infinitesimal neighbourhood of $E$ in $Y$ and, for each $n\ge 1$, consider the exact sequence
\begin{align}
\label{ses_egl}
    0 \rightarrow {\rm Sym}^n(M\oplus (M\otimes L)) \rightarrow \cal O_{Y_n} \rightarrow \cal O_{Y_{n-1}} \rightarrow 0.
\end{align}
Observe that $E$ is Gorenstein and $\omega_E = \cal O_E$.
Thus, 
\[h^1(E, {\rm Sym}^n(M\oplus (M\otimes L))) = h^0(E, {\rm Sym}^n(M\oplus (M\otimes L))^*) = 0,\]
where this latter equality holds because $M$ and $M\otimes L$ are ample.
It follows that the short exact sequence \eqref{ses_egl} splits for all $n \ge 1$ and so by induction on $n$ we have   that $\cal O_{\widehat{Y}} = \bigoplus_{n \geq 0} {\rm Sym}^n(M\oplus (M\otimes L))$. Since $\cal O_Y(D)\vert_E = M^{*}$, if we denote
\[
\cal A_{m,n} \coloneqq H^0(E, {\rm Sym}^n(M\oplus (M\otimes L))\otimes M^{-m})
\]
then $\mathcal A$ is bi-graded, and
\[
\cal A \simeq \bigoplus_{m,n\ge 0} \mathcal A_{m,n}.
\]
We have
\[ 
\mathcal A_{m,n}\simeq \bigoplus_{k=0}^n H^0(E, M^{n-m}\otimes L^k).
\]
Note that $\cal A_{m, n} = 0$ for $n < m$ and $\cal A_{m, n} \simeq H^0(E, \cal O_E)$
for $n = m$.

Assume by contradiction that $\mathcal A$ is finitely generated, and let
$\sigma_1,\ldots,\sigma_t$ be homogeneous generators. For each non-zero
homogeneous component of a generator lying in
$H^0(E,M^a\otimes L^b)$ with $a>0$, consider the ratio $b/a$, and let
$\ell$ be the maximum of these finitely many ratios. Since there are no
non-zero sections of $L^b$ for $b>0$, every non-zero homogeneous element
with $b>0$ has $a>0$. Hence every element generated by the $\sigma_i$ has
slope $b/a\leq \ell$. On the other hand, for any $a'>0$ and $b'>\ell a'$, the line bundle
$M^{a'}\otimes L^{b'}$ is ample, and if $a' \gg 0$ we have
\(
H^0(E,M^{a'}\otimes L^{b'})\neq 0.
\)
This gives a non-zero element of $\mathcal A$ of slope $b'/a'>\ell$, a contradiction.

Thus, we have proven:

\begin{theorem}\label{t_flopsdonotexist}
There exist
\begin{enumerate}
\item  a smooth threefold $Y$, 
\item a rank one foliation $\cal H$  on $Y$   with canonical singularities,
\item  a flopping contraction $f\colon Y\to Z$ for $\cal H$, and 
\item a divisor $D$ on $Y$
\end{enumerate}
such that $-D$ is $f$-ample but  the flopping contraction $f\colon Y\to Z$ does not admit a $D$-flop. 
\end{theorem}

\subsection{A rank one foliation without a canonical model}
\label{s_no_canonical_model}
We use the same set-up as in \S \ref{s_setup}.
Set $\cal N \coloneqq \cal O_X(-K_{\cal G} )$ so that we have an exact sequence 
\[
0 \rightarrow \cal N^* \rightarrow \Omega^1_X(\log E)\rightarrow \mathcal O_X( K_{\cal F}) \rightarrow 0.
\]
Let $P \coloneqq \mathbb P(\mathcal N\oplus \cal O_X)$ with projection $p\colon P \rightarrow X$ and let $X_0$ and $X_\infty$ be the corresponding sections
so that the normal sheaves of $X_0$ and $X_{\infty}$ in $P$ are given by 
$\cal N_{X_0/P} =\cal N$ and $\cal N_{X_\infty/P} =\cal N^*$ respectively.

Let $V \coloneqq P \setminus X_{\infty}$. We proceed similarly as in the previous section to construct a foliation on $V$. 
Let $\{U_i\}$ be an open cover of $X$ so that $\cal F$ is defined by a logarithmic  $1$-form $\omega_i$ on $U_i$
with poles along $E$. Note that $\{\omega_i\}$ gives an element of $H^0(X, \Omega^1_X(\log E)\otimes \cal N)$.
We have $p^{-1}(U_i) \cap V \simeq U_i \times \bb C$ and we denote by $z_i$ the corresponding coordinate on $\bb C$.
Let $h_{ij}$ be the transition function of the line bundle associated to $\cal N$, with respect to the open cover $\{U_i\}$, so that $\omega_{i} = h_{ij}\omega_j$  and  $z_i = h_{ij}z_j$, for all $i,j$. 

Fix $N\ge 2$. We want to define a collection of logarithmic 2-forms $\Omega_i$ on $p^{-1}(U_i)\cap V$  
which satisfy $\Omega_i = h_{ij}^2\Omega_j$, which have poles along $p^{-1}(E)$ and which define a rank one foliation on $P$ which is tangent to the fibration $P \to X$ along the divisor $X_{\infty}$ and the order of this tangency is $N-2$.

Fix an integer $n\ge 2$. By Lemma  \ref{lem_hmf_calcs}, there exists a non-zero section of $H^0(X,\Omega^1_X(\log E) \otimes (\mathcal{N}^*)^{\otimes n-1})$.
This defines a logarithmic 1-form $\eta_{in}$ on $U_i$ with poles on $E$ and such that 
\[\eta_{jn} = h_{ij}^{n-1}\eta_{in}.\]

On $p^{-1}(U_i)\cap V$ we define
\[\Omega_i \coloneqq z_i^2d\left(\frac{\omega_i}{z_i}\right) + \sum_{n=2}^N z_i^n \omega_i \wedge \eta_{in}.\]

It follows by a direct computation that $\Omega_i=h_{ij}^2\Omega_j$ and that the $2$-forms $\Omega_i$ define a rank one foliation on $V$, which extends to a foliation $\cal H$ on $P$.

Let $x \in X$ be a general point, and let $F$ be the fibre over $x$.  Let $U = U_i$ be a neighbourhood
of $x$ as defined above and let $t_i = \frac{1}{z_i}$.
In a neighbourhood of $F \cap X_{\infty}$, we have that
$\cal H$ is generated by \[ -t_i^{N-2}\omega_i\wedge dt_i+t_i^{N-1}d\omega_i+\sum_{n=2}^N t_i^{N-n} \omega_i\wedge \eta_{in}.\]

Since there are no $K_{\cal H}$-negative invariant rational
curves, \cite[Theorem 4.8]{CS23} implies that $K_{\cal H}$ is nef. Thus,
the foliation  $\cal H$ satisfies the following properties:
\begin{enumerate}
\item $X_0$ is $\cal H$-invariant and $\cal H$ restricted to $X_0$ coincides with $\cal F$;

\item $\cal H$ is everywhere transverse to $F$ except at $F \cap X_{\infty}$ where it acquires a tangency and, 
in particular, $K_{\cal H}\cdot F = N-2$; 

\item $\cal H$ is generically transverse to $X_\infty$;

\item \(
K_{\mathcal H}\sim p^*K_{\mathcal F}+(N-2)X_\infty
\)

 is nef and big for any $N\ge 3$. 
\end{enumerate}

In particular, we have that $K_{\cal H}|_{X_0}\simeq K_{\cal F}$ is not big and, therefore $X_0$ is contained in  $\Null K_{\cal H}$.
Observe that there is no morphism $f\colon X_0 \rightarrow \Sigma$ where $\dim \Sigma \leq 1$
so that $\cal N^*_{X_0/P} = \cal N^*_{\cal F}(E)$ is $f$-ample.  In particular, there does not exist a birational morphism 
$c\colon P\to Z$ onto an algebraic space $Z$  whose exceptional locus contains $X_0$.

To summarise, we have proven:

\begin{theorem}\label{t_canonicalmodelsdonotexist}
There exist
\begin{enumerate}
\item a smooth projective threefold $P$, and
\item a rank one foliation $\cal H$ on $P$ with canonical singularities
\end{enumerate}
such that $K_{\cal H}$ is big and nef, but $\cal H$ does not admit a canonical model, in the sense that there does not exist a morphism
\(
c\colon P\to Z
\)
in the category of algebraic spaces such that $c$ contracts $\Null K_{\cal H}$ to a point or a curve.
\end{theorem}

\bibliography{math.bib}
\bibliographystyle{alpha}

\end{document}